\documentclass[11pt]{amsart}
\usepackage[latin1]{inputenc}
\usepackage[T1]{fontenc}
\usepackage{hyperref}
\usepackage{amsfonts}
\usepackage{amsmath}
\usepackage{epsf, subfigure, verbatim}
\usepackage{amssymb}
\usepackage{amsthm}
\usepackage{latexsym}
\usepackage{layout}
\usepackage{amsrefs}

\newtheorem{theorem}{Theorem}

\newtheorem{lemma}{Lemma}[section]
\newtheorem{corollary}{Corollary}
\newtheorem{assumption}{Assumption}[section]
\newtheorem{remark}{Remark}[section]

\newcommand{\reals}{\mathbb{R}}
\newcommand{\ind}{\mathbf{1}}
\newcommand{\e}{\mathbb{E}}
\newcommand{\p}{\mathbb{P}}
\newcommand{\exptime}{\mathbf{e}_\alpha}
\begin{document}

\title[]{Pricing occupation-time options in a mixed-exponential jump-diffusion model}

\author[]{Djilali Ait Aoudia}
\address{Quantact \& D\'epartement de math\'ematiques, Universit\'e du Qu\'ebec \`a Montr\'eal (UQAM), 201 av.\ Pr\'esident-Kennedy, Montr\'eal (Qu\'ebec) H2X 3Y7, Canada}
\email{ait\_aoudia.djilali@courrier.uqam.ca}

\author[]{Jean-Fran\c{c}ois Renaud}
\address{D\'epartement de math\'ematiques, Universit\'e du Qu\'ebec \`a Montr\'eal (UQAM), 201 av.\ Pr\'esident-Kennedy, Montr\'eal (Qu\'ebec) H2X 3Y7, Canada}
\email{renaud.jf@uqam.ca}

\date{\today}


\begin{abstract}
In this short paper, in order to price occupation-time options, such as (double-barrier) step options and quantile options, we derive various joint distributions of a mixed-exponential jump-diffusion process and its occupation times of intervals.
\end{abstract}

\maketitle

\section{Introduction}

Let the price of an (underlying) asset $S = \{ S_t , t \geq 0\}$ be of the form:
$$
S_t = S_0 \mathrm{e}^{X_t} ,
$$
where $X = \{ X_t , t \geq 0\}$ is a process to be specified (log-return process). For example, in the Black-Scholes-Merton (BSM) model, $X$ is a Brownian motion with drift. The time spent by $S$ in an interval $I$, or equivalently the time spent by $X$ in an interval $I^\prime$, from time $0$ to time $T$, is given by
$$
A_T^I := \int_0^T \mathbf{1}_{\{S_t \in I\}} \mathrm{d}t = \int_0^T \mathbf{1}_{\{X_t \in I^\prime\}} \mathrm{d}t .
$$
Options linked to occupation times are often seen as generalized barrier options. Instead of being activated (or canceled) when the underlying asset price crosses a barrier, which is a problem from a risk management point of view, the payoff of occupation-time options will depend on the time spent above/below this barrier: the change of value occurs more gradually. There are several different options: (barrier) step options, corridor derivatives, cumulative(-boost) options, quantile options, (cumulative) Parisian options, etc. For a review, see e.g.\ \cite{pechtl_1999}. 

Introduced by Linetsky \cite{linetsky_1999}, a (down-and-out call) \textit{step option} admits the following payoff:
$$
\mathrm{e}^{- \rho A_T^{L,-} } \left( S_T - K \right)_+ = \mathrm{e}^{- \rho A_T^{L,-} } \left( S_0 \mathrm{e}^{X_T} - K \right)_+ ,
$$
where
$$
A_T^{L,-} = \int_0^T \mathbf{1}_{\{S_t \leq L\}} \mathrm{d}t ,
$$
and where $\rho > 0$ is called the \textit{knock-out rate}. Indeed, it is interesting to note that we have the following relationship:
$$
\mathbf{1}_{\{\tau_L^- > T\}} \left( S_T - K \right)_+  \leq \mathrm{e}^{- \rho A_T^{L,-} } \left( S_T - K \right)_+ \leq  \left( S_T - K \right)_+ .
$$
where $\tau_L^- = \inf \{ t \geq 0 \colon S_t \leq L \}$. Later, Davydov and Linetsky \cite{davydov_linetsky_2002} studied double-barrier step call options, which are a generalization of double-barrier options (see \cite{geman_yor_1996}):
$$
\mathrm{e}^{- \rho^- A_T^{L,-} - \rho^+ A_T^{U,+}} \left( S_T - K \right)_+ = \mathrm{e}^{- \rho^- A_T^{L,-} - \rho^+ A_T^{U,+} } \left( S_0 \mathrm{e}^{X_T} - K \right)_+ ,
$$
where
$$
A_T^{U,+} = \int_0^T \mathbf{1}_{\{S_t \geq U\}} \mathrm{d}t ,
$$
and where $\rho^-$ and $\rho^+$ are the knock-out rates. Expressions for the price of double-barrier step options are available in the BSM model and for single-barrier step options in Kou's model for example.

Studied by Fusai \cite{fusai_2000} in the BSM model (see also the work of Akahori and Tak\`acs), a \textit{corridor option} admits the following payoff: for $K < T$,
$$
\left( A_T^{L,U} - K \right)_+ = \left( \int_0^T \mathbf{1}_{\{h < X_s < H\}} \mathrm{d}s - K \right)_+ .
$$
If $h = -\infty$, it is called a \textit{hurdle option}. The distribution of this occupation time is linked to L\'evy's arc-sine law in the BSM model. Again, expressions for the price of double-barrier corridor options are available in the BSM model and for single-barrier barrier options in Kou's model.

Miura \cite{miura_1992} introduced $\alpha$-quantile options as an extension of lookback options. The $\alpha$-quantile of the log-return process $X$ is defined, for $0 \leq \alpha \leq 1$, by
$$
q(\alpha,T) := \inf \left\lbrace h \colon \int_0^T \mathbf{1}_{\{X_t \leq h\}} \mathrm{d}t > \alpha T \right\rbrace .
$$
A fixed-strike $\alpha$-quantile call option admits the following payoff:
$$
\left( S_0 \mathrm{e}^{\gamma q(\alpha,T)}  - K \right)_+ .
$$
When $\alpha=0$ and $\gamma=1$, the quantile option is reduced to a lookback option. Indeed, when $\alpha = 0$,
$$
q(0,T) = \sup_{0 \leq t \leq T} X_t .
$$

In summary, in order to price many of these options, we are interested in the joint distribution of
$$
\left( \int_0^T \mathbf{1}_{\{L < S_t < U\}} \mathrm{d}t , S_T \right) ,
$$
or equivalently,
$$
\left( \int_0^T \mathbf{1}_{\{h < X_t < H\}} \mathrm{d}t , X_T \right) ,
$$
where $h=\ln(L/S_0)$ and $H=\ln(U/S_0)$.

In Black-Scholes-Merton model, in the constant elasticity of variance (CEV) model and in Kou's model, the \textit{standard} technique for deriving this joint distribution (joint Laplace transform) has been to use the Feynman-Ka\v{c} formula; see \cite{hugonnier_1999}, \cite{leung_kwok_2007} and \cite{cai_chen_wan_2010}.

Our goal is to price occupation-time options. In doing so, we extend results previously obtained in a nice paper by Cai, Chen and Wan \cite{cai_chen_wan_2010}. We extend their results in two directions: by looking simultaneously at more general functionals of occupation times and a more general jump-diffusion process. For example, in order to price double step options, we derive the joint distribution of
$$
\left( \int_0^T \mathbf{1}_{\{X_t < h\}} \mathrm{d}t , \int_0^T \mathbf{1}_{\{X_t > H\}} \mathrm{d}t , X_T \right) .
$$
We develop a \textit{probabilistic approach} to obtain these distributions in a mixed-exponential jump-diffusion model (MEM), an approach often refered to as the \textit{perturbation approach} and which is in the spirit of G\'eman \& Yor \cite{geman_yor_1996}; it is based on a decomposition of the trajectories of the underlying (log-return) process using the solutions to the one-sided and the two-sided exit problems. Our methodology uses extensions of results developed by Cai and Kou \cite{cai_kou_2011} (see also \cite{chen_sheu_chang_2013}). Finally, we answer several open questions from \cite{cai_chen_wan_2010}; see e.g.\ the first paragraph on p.\ 434 and our Lemma~\ref{L:inversion}.

In summary, the contributions of this paper consist in new probabilistic derivations of several joint Laplace(-Carson) transforms of a mixed-exponential jump-diffusion process and its occupation times of an interval, all sampled at a fixed time, and the pricing of occupation-time derivatives such as (double-barrier) step options and $\alpha$-quantile options in a mixed-exponential model. The objectives are as in \cite{cai_chen_wan_2010}, but in a more general model and with a different methodology.
%
%

The rest of the paper is organized as follows. Section $2$ introduces the mixed-exponential jump-diffusion process and some of its elementary properties. In Section 3, we present our main theoretical results on occupation times involving the mixed-exponential jump-diffusion process; the proofs are left for the Appendix. Finally, in Section 4, we use the results of Section 3 to derive Laplace transforms of the price for various occupation-time options.

\section{The mixed-exponential jump-diffusion process}

A L\'evy jump-diffusion process $X=\{X_t, t\geq0 \}$ is defined as
$$
X_t = X_0 + \mu t + \sigma W_t + \sum_{i=1}^{N_t} Y_i ,
$$
where $\mu \in \mathbb{R}$ and $\sigma\geq 0$ represent the drift and volatility of the diffusion part respectively, $W=\{W_t, t\geq0 \}$ is a (standard) Brownian motion, $N=\{N_t, t\geq0 \}$ is a homogeneous Poisson process with rate $\lambda$ and $\{Y_i, i=1,2,\dots \}$ are independent and identically distributed random variables. These quantities are mutually independent. When $\sigma>0$, the infinitesimal generator of $X$ acts on functions $h \in \mathcal{C}_{0}^{2}(\mathbb{R})$ and is given by
\begin{equation}\label{E:generator}
\mathcal{L}h(x) = \mu h^\prime(x) + \frac{\sigma^2}{2} h^{\prime \prime}(x) + \lambda \int_{-\infty}^{\infty} \left( h(x+y)-h(x) \right) f_Y(y) \mathrm{d}y .
\end{equation}
When $\sigma=0$, the function $h$ needs only to be once differentiable.

In a jump-diffusion market model, the dynamic of the asset price $S$ is given under the risk-neutral measure $\mathbb{P}$ by:
$$
\frac{\mathrm{d}S_t}{S_{t-}} = r \mathrm{d}t + \sigma \mathrm{d}W_t + \mathrm{d}\left( \sum_{i=1}^{N_t} \left( \mathrm{e}^{Y_i}-1 \right) \right) ,
$$
where $r>0$ is the risk-free rate. Solving this stochastic differential equation, one obtains
$$
S_t = S_0 \mathrm{e}^{X_t} = S_0 \exp \left\lbrace \mu t + \sigma W_t + \sum_{i=1}^{N_t} Y_i \right\rbrace ,
$$
where $\mu = r - \sigma^2/2 - \lambda \left( \mathbb{E} \left[ \mathrm{e}^{Y_1} \right] - 1 \right)$. Clearly, for that purpose, we will need to assume that the $Y_i$'s have a finite moment generating function.

In the pioneer work of Merton \cite{merton_1976}, the common distribution of the $Y_i$'s is chosen to be a normal distribution, while in \cite{kou_2002} it is a double-exponential distribution, i.e.\ the common probability density function (pdf) is given by
$$
f_Y (y) = p \eta \mathrm{e}^{-\eta y} \mathbf{1}_{\{y \geq 0\}} + (1-p) \theta \mathrm{e}^{\theta y} \mathbf{1}_{\{ y<0 \}} ,
$$
where $0<p<1$, $\eta>0$ and $\theta>0$. When the jumps sizes are hyper-exponentially distributed, their common pdf is given by
\begin{equation}\label{E:pdf_jumps_hem}
f_Y (y) = \sum_{i=1}^{m} p_{i} \eta_i \mathrm{e}^{-\eta_i y} \mathbf{1}_{\{y \geq 0\}} + \sum_{j=1}^{n} q_j \theta_j \mathrm{e}^{\theta_j y} \mathbf{1}_{\{ y<0 \}} ,
\end{equation}
where $p_i, q_j > 0$ for all $i=1,2,\dots,m$ and $j=1,2,\dots,n$, and such that $\sum_{i=1}^{m}p_i+\sum_{j=1}^{n}q_i=1$, and where $\eta_1 < \ldots < \eta_m$ and $\theta_1 < \ldots < \theta_n$, then $X$ is said to be a hyper-exponential jump-diffusion (HEJD) process and the market model is called the hyper-exponential model (HEM).

We will use a slightly more general, and thus more flexible, jump distribution: the mixed-exponential distribution. In this case, the common pdf is given by
\begin{equation}\label{E:pdf_jumps_mem}
f_Y (y) = p_u \sum_{i=1}^{m} p_{i} \eta_i \mathrm{e}^{-\eta_i y} \mathbf{1}_{\{y \geq 0\}} + q_d \sum_{j=1}^{n} q_j \theta_j \mathrm{e}^{\theta_j y} \mathbf{1}_{\{ y<0 \}} ,
\end{equation}
where $p_u, q_d \geq 0$ and $p_u+q_d=1$, where now $p_i, q_j \in (-\infty, \infty)$ for all $i=1,2,\dots,m$ and $j=1,2,\dots,n$, such that $\sum_{i=1}^{m}p_i = 1$, $\sum_{j=1}^{n}q_i=1$, and where again $\eta_1 < \ldots < \eta_m$ and $\theta_1 < \ldots < \theta_n$. The resulting jump-diffusion process $X$ is said to be a mixed-exponential jump-diffusion (MEJD) process and the market model is called the mixed-exponential model (MEM).

The MEM is a financial model fitting the data quite well and still being very tractable; for more information, see \cite{cai_kou_2011}. One of the main feature is probably that the mixed-exponential distribution can approximate any jump distribution (in the sense of weak convergence). See the paper by Cai and Kou \cite{cai_kou_2011} for more information on this process.

Throughout the rest of the paper, the law of $X$ such that $X_0 = x$ is denoted by $\p_x$ and the corresponding expectation by $\e_x$; we write $\p$ and $\e$ when $x=0$. The L\'evy exponent of a MEJD $X$ is given by
\begin{align*}
G(\zeta) &= \frac{\ln \mathbb{E} \left[ \exp(\zeta X_t) \right]}{t}\\
&= \mu \zeta + \frac{\sigma^2}{2} \zeta^2 +\lambda \left( p_u \sum_{i=1}^{m} \frac{p_i \eta_i}{\eta_i-\zeta} + q_d \sum_{i=1}^{n} \frac{q_i \theta_i}{\theta_i + \zeta} - 1 \right) ,
\end{align*}
for any $\zeta \in (-\theta_1,\eta_1)$. Then, clearly, the trend of the process is given by
$$
\e \left[ X_1 \right] = G'(0+) = \mu + \lambda \left( p_u\sum_{i=1}^m \frac{p_i}{\eta_i} - q_d \sum_{j=1}^m \frac{q_j}{\theta_j} \right) .
$$

For any $\alpha \in \mathbb{R}$, the function $\zeta \mapsto G(\zeta)-\alpha$ has at most $n+m+2$ real roots. It can be shown (see \cite{cai_kou_2011}*{Theorem 3.1}) that, for a sufficiently large $\alpha>0$, the corresponding Cram\'er-Lundberg equation $G(\zeta)=\alpha$ has exactly $n+m+2$ distinct real roots; there are $m+1$ positive roots denoted by $\beta_{1,\alpha},\ldots,\beta_{m+1,\alpha}$ and $n+1$ negative roots $\gamma_{1,\alpha},\ldots,\gamma_{n+1,\alpha}$, satisfying 
\begin{gather*}
0 < \beta_{1,\alpha} < \eta_1 < \beta_{2,\alpha} < \ldots < \eta_m < \beta_{m+1,\alpha} < \infty ,\\
-\infty < \gamma_{n+1,\alpha} < -\theta_n < \gamma_{n,\alpha} < \ldots < \gamma_{2,\alpha} < \theta_1 < \gamma_{1,\alpha} < 0 .
\end{gather*}
Finally, let $S=m+n+2$ and define $\overrightarrow{\rho_\alpha} = \left( \rho_{1,\alpha},\dots,\rho_{S,\alpha} \right) = \left( \beta_{1,\alpha},\ldots,\beta_{m+1,\alpha}, \gamma_{1,\alpha},\ldots,\gamma_{n+1,\alpha} \right)$, the vector containing all the roots.
%

\begin{assumption}\label{A:roots}
For the rest of the paper, we assume that, for a given value of $\alpha$, the Cram\'er-Lundberg equation $G(\zeta)=\alpha$ has exactly $n+m+2$ distinct real solutions as above.
\end{assumption}

\begin{remark}
In the case of a HEJD, a detailed study of the roots is undertaken in \cite{cai_2009}.
\end{remark}

\subsection{First passage and two-sided exit problems}

For $b \in \mathbb{R}$, define the first passage times $\tau^{+}_b = \inf \{t\geq 0 \colon X_t > b \}$ and $\tau^{-}_b=\inf\{ t \geq 0 \colon X_t < b \}$, with the convention $\inf \emptyset=\infty$.

Consider two barrier levels $h$ and $H$ such that $h<H$. It has been shown in \cite{cai_kou_2011}*{Theorem 3.3} that, for any sufficiently large $\alpha > 0$, $\theta<\eta_1$ and $x < h$,
$$
\e_x \left[ \mathrm{e}^{-\alpha \tau_h^{+} + \theta X_{\tau_h^{+}} } \right] = \sum_{i=1}^{m+1} c_i \mathrm{e}^{\beta_{i,\alpha} x} ,
$$
where $(c_1, c_2, \dots, c_{m+1})$ is a vector of constants (uniquely determined by a nonsingular linear system). It should be pointed out here that the coefficients depend on $\alpha$, $\theta$ and $h$, as well as the parameters of the process (explicitly and implicitly).

Recently, in \cite{chen_sheu_chang_2013}*{Theorem 2.5}, a very similar result using the same method of proof as in \cite{cai_kou_2011}*{Theorem 3.3} has been obtained for the case of a HEJD (not a MEJD) process. It has been shown that, for $\alpha \geq 0$ (and for $\sigma > 0$), a nonnegative bounded function $g(\cdot)$ on $(h,H)^c$ and $h < x < H$, 
$$
\e_x \left[ \mathrm{e}^{-\alpha \left( \tau_H^{+} \wedge \tau_h^{-} \right) } g \left( X_{\tau_H^{+} \wedge \tau_h^{-}} \right) \right] = \sum_{i=1}^{S} b_{i} \mathrm{e}^{\rho_{i} x} ,
$$
where $(b_1, b_2, \dots, b_{S})$ is a vector of constants also to be determined.
%

We now provide slight extensions of the abovementioned results. The proof is left to the reader; it follows the same steps as in the proof of \cite{cai_kou_2011}*{Theorem 3.3} and \cite{chen_sheu_chang_2013}*{Theorem 2.5}.

\begin{theorem}\label{th1}
Let $X$ be a MEJD process. Under Assumption~\ref{A:roots} for a given value of $\alpha > 0$, and for a nonnegative and bounded real-valued function $g(\cdot)$, we have:
\begin{enumerate}
\item for $x < H$,
$$
\e_x \left[ \mathrm{e}^{-\alpha \tau_H^{+}} g \left( X_{\tau_H^{+}} \right)  \right] = \sum_{i=1}^{m+1} \omega_i \mathrm{e}^{\beta_{i,\alpha} x} ,
$$
where $\overrightarrow{\omega} = (\omega_1, \omega_2, \dots, \omega_{m+1})$ is a vector (uniquely) determined by the following linear system:
$$
A^{H,\alpha} \overrightarrow{\omega} = J^{H,g} ,
$$
where $A^{H,\alpha}$ is an $(m+1) \times (m+1)$ nonsingular matrix given by
$$
A^{H,\alpha} = \left(
\begin{array}{cccc}
\mathrm{e}^{\beta_{1,\alpha} H} & \mathrm{e}^{\beta_{2,\alpha} H} & \dots & \mathrm{e}^{\beta_{m+1,\alpha} H} \\
\frac{\mathrm{e}^{\beta_{1,\alpha} H}}{\eta_1-\beta_{1,\alpha}} & \frac{\mathrm{e}^{\beta_{2,\alpha} H}}{\eta_1-\beta_{2,\alpha}} & \dots & \frac{\mathrm{e}^{\beta_{m+1,\alpha} H}}{\eta_1-\beta_{m+1,\alpha}} \\
\frac{\mathrm{e}^{\beta_{1,\alpha} H}}{\eta_2-\beta_{1,\alpha}} & \frac{\mathrm{e}^{\beta_{2,\alpha} H}}{\eta_2-\beta_{2,\alpha}} & \dots & \frac{\mathrm{e}^{\beta_{m+1,\alpha} H}}{\eta_2-\beta_{m+1,\alpha}} \\
\vdots & \vdots & \ddots & \vdots \\
\frac{\mathrm{e}^{\beta_{1,\alpha} H}}{\eta_m-\beta_{1,\alpha}} & \frac{\mathrm{e}^{\beta_{2,\alpha} H}}{\eta_m-\beta_{2,\alpha}} & \dots & \frac{\mathrm{e}^{\beta_{m+1,\alpha} H}}{\eta_m-\beta_{m+1,\alpha}} \\
\end{array}
\right) ,
$$
and where $J^{H,g}$ is an $(m+1)$-dimensional vector given by
$$
\left( g(H+) , \mathrm{e}^{\eta_1 H} \int_H^\infty g(y) \mathrm{e}^{-\eta_1 y} \mathrm{d}y, \mathrm{e}^{\eta_2 H} \int_H^\infty g(y) \mathrm{e}^{-\eta_2 y} \mathrm{d}y , \dots , \mathrm{e}^{\eta_m H} \int_H^\infty g(y) \mathrm{e}^{-\eta_m y} \mathrm{d}y \right) .
$$

\item for $x > h$,
$$
\e_x \left[ \mathrm{e}^{-\alpha \tau_h^{-}} g \left( X_{\tau_h^{-}} \right) \right] = \sum_{i=1}^{n+1} \nu_i \mathrm{e}^{\gamma_{i,\alpha} x} ,
$$
where $\overrightarrow{\nu} = (\nu_1, \nu_2, \dots, \nu_{n+1})$ is a vector (uniquely) determined by the following linear system:
$$
A^{h,\alpha} \overrightarrow{\mathbf{\nu}} = J^{h,g} ,
$$
where $A^{h,\alpha}$ is an $(n+1) \times (n+1)$ nonsingular matrix given by
$$
A^{h,\alpha} = \left(
\begin{array}{cccc}
\mathrm{e}^{\gamma_{1,\alpha} h} & \mathrm{e}^{\gamma_{2,\alpha} h} & \dots & \mathrm{e}^{\gamma_{n+1,\alpha} h} \\
\frac{\mathrm{e}^{\gamma_{1,\alpha} h}}{\theta_1+\gamma_{1,\alpha}} & \frac{\mathrm{e}^{\gamma_{2,\alpha} h}}{\theta_1+\gamma_{2,\alpha}} & \dots & \frac{\mathrm{e}^{\gamma_{n+1,\alpha} h}}{\theta_1+\gamma_{n+1,\alpha}} \\
\frac{\mathrm{e}^{\gamma_{1,\alpha} h}}{\theta_2+\gamma_{1,\alpha}} & \frac{\mathrm{e}^{\gamma_{2,\alpha} h}}{\theta_2+\gamma_{2,\alpha}} & \dots & \frac{\mathrm{e}^{\gamma_{n+1,\alpha} h}}{\theta_2+\gamma_{n+1,\alpha}} \\
\vdots & \vdots & \ddots & \vdots \\
\frac{\mathrm{e}^{\gamma_{1,\alpha} h}}{\theta_n+\gamma_{1,\alpha}} & \frac{\mathrm{e}^{\gamma_{2,\alpha} h}}{\theta_n+\gamma_{2,\alpha}} & \dots & \frac{\mathrm{e}^{\gamma_{n+1,\alpha} h}}{\theta_n+\gamma_{n+1,\alpha}} \\
\end{array}
\right) ,
$$
and where $J^{h,g}$ is an $(n+1)$-dimensional vector given by
$$
\left( g(h-) , \mathrm{e}^{-\theta_1 h} \int_{-\infty}^h g(y) \mathrm{e}^{\theta_1 y} \mathrm{d}y , \mathrm{e}^{-\theta_2 h} \int_{-\infty}^h g(y) \mathrm{e}^{\theta_2 y} \mathrm{d}y , \dots , \mathrm{e}^{-\theta_n h} \int_{-\infty}^h g(y) \mathrm{e}^{\theta_n y} \mathrm{d}y \right) .
$$
\item for $h<x<H$,
$$
\e_x \left[ \mathrm{e}^{-\alpha \left( \tau_H^{+} \wedge \tau_h^{-} \right) } g \left( X_{\tau_H^{+} \wedge \tau_h^{-}} \right) \right] = \sum_{i=1}^{m+1} \omega_i \mathrm{e}^{\beta_{i,\alpha} x} + \sum_{i=1}^{n+1} \nu_i \mathrm{e}^{\gamma_{i,\alpha} x} = \sum_{i=1}^{S} Q_{i} \mathrm{e}^{\rho_{i, \alpha} x} ,
$$
where $\mathbf{Q} = (Q_1, Q_2, \dots, Q_{S}) = (\omega_1, \omega_2, \dots, \omega_{m+1}, \nu_1, \nu_2, \dots, \nu_{n+1})$ is a vector (uniquely) determined by the following linear system:
$$
A^{h,H,\alpha} \mathbf{Q} = J^{h,H,g} ,
$$
where $A^{h,H,\alpha}$ is an $S \times S$ nonsingular matrix given by
\begin{align*}
A^{h,H,\alpha} &= \left(
\begin{array}{cccccc}
\mathrm{e}^{\beta_{1,\alpha} h} & \dots & \mathrm{e}^{\beta_{m+1,\alpha} h} & \mathrm{e}^{\gamma_{1,\alpha} h} & \dots & \mathrm{e}^{\gamma_{n+1,\alpha} h} \\
\mathrm{e}^{\beta_{1,\alpha} H} & \dots & \mathrm{e}^{\beta_{m+1,\alpha} H} & \mathrm{e}^{\gamma_{1,\alpha} H} & \dots & \mathrm{e}^{\gamma_{n+1,\alpha} H} \\
\frac{\mathrm{e}^{\beta_{1,\alpha} h}}{\theta_1+\beta_{1,\alpha}} & \dots & \frac{\mathrm{e}^{\beta_{m+1,\alpha} h}}{\theta_1+\beta_{m+1,\alpha}} & \frac{\mathrm{e}^{\gamma_{1,\alpha} h}}{\theta_1+\gamma_{1,\alpha}} & \dots &\frac{\mathrm{e}^{\gamma_{n+1,\alpha} h}}{\theta_1+\gamma_{n+1,\alpha}} \\
\vdots & \vdots & \ddots & \vdots & \vdots \\
\frac{\mathrm{e}^{\beta_{1,\alpha} h}}{\theta_n+\beta_{1,\alpha}} & \dots & \frac{\mathrm{e}^{\beta_{m+1,\alpha} h}}{\theta_n+\beta_{m+1,\alpha}} & \frac{\mathrm{e}^{\gamma_{1,\alpha} h}}{\theta_n+\gamma_{1,\alpha}} & \dots & \frac{\mathrm{e}^{\gamma_{n+1,\alpha} h}}{\theta_n+\gamma_{n+1,\alpha}} \\
\frac{\mathrm{e}^{\beta_{1,\alpha} H}}{\eta_1-\beta_{1,\alpha}} & \dots & \frac{\mathrm{e}^{\beta_{m+1,\alpha} H}}{\eta_1-\beta_{m+1,\alpha}} & \frac{\mathrm{e}^{\gamma_{1,\alpha} H}}{\eta_1-\gamma_{1,\alpha}} & \dots & \frac{\mathrm{e}^{\gamma_{n+1,\alpha} H}}{\eta_1-\gamma_{n+1,\alpha}} \\
\vdots & \vdots & \ddots & \vdots & \vdots \\
\frac{\mathrm{e}^{\beta_{1,\alpha} H}}{\eta_m-\beta_{1,\alpha}} & \dots & \frac{\mathrm{e}^{\beta_{m+1,\alpha} H}}{\eta_m-\beta_{m+1,\alpha}} & \frac{\mathrm{e}^{\gamma_{1,\alpha} H}}{\eta_m-\gamma_{1,\alpha}} & \dots & \frac{\mathrm{e}^{\gamma_{n+1,\alpha} H}}{\eta_m-\gamma_{n+1,\alpha}} \\
\end{array}
\right)\\
&= \left(
\begin{array}{cccc}
\mathrm{e}^{\rho_{1,\alpha} h} & \mathrm{e}^{\rho_{2,\alpha} h} & \dots & \mathrm{e}^{\rho_{S,\alpha} h} \\
\mathrm{e}^{\rho_{1,\alpha} H} & \mathrm{e}^{\rho_{2,\alpha} H} & \dots & \mathrm{e}^{\rho_{S,\alpha} H} \\
\frac{\mathrm{e}^{\rho_{1,\alpha} h}}{\theta_1+\rho_{1,\alpha}} & \frac{\mathrm{e}^{\rho_{2,\alpha} h}}{\theta_1+\rho_{2,\alpha}} & \dots & \frac{\mathrm{e}^{\rho_{S,\alpha} h}}{\theta_1+\rho_{S,\alpha}} \\
\vdots & \vdots & \ddots & \vdots \\
\frac{\mathrm{e}^{\rho_{1,\alpha} h}}{\theta_n+\rho_{1,\alpha}} & \frac{\mathrm{e}^{\rho_{2,\alpha} h}}{\theta_n+\rho_{2,\alpha}} & \dots & \frac{\mathrm{e}^{\rho_{S,\alpha} h}}{\theta_n+\rho_{S,\alpha}} \\
\frac{\mathrm{e}^{\rho_{1,\alpha} H}}{\eta_1-\rho_{1,\alpha}} & \frac{\mathrm{e}^{\rho_{2,\alpha} H}}{\eta_1-\rho_{2,\alpha}} & \dots & \frac{\mathrm{e}^{\rho_{S,\alpha} H}}{\eta_1-\rho_{S,\alpha}} \\
\vdots & \vdots & \ddots & \vdots \\
\frac{\mathrm{e}^{\rho_{1,\alpha} H}}{\eta_m-\rho_{1,\alpha}} & \frac{\mathrm{e}^{\rho_{2,\alpha} H}}{\eta_m-\rho_{2,\alpha}} & \dots & \frac{\mathrm{e}^{\rho_{S,\alpha} H}}{\eta_m-\rho_{S,\alpha}} \\
\end{array}
\right) ,
\end{align*}
and where $J^{h,H,g}$ is an $S$-dimensional vector given by
$$
J^{h,H,g} = \left( J^{H,g} , J^{h,g} \right) .
$$
\end{enumerate}
\end{theorem}

Clearly, by the definition of the stopping times, if $x > h$ (resp.\ $x < H$), then
$$
\e_x \left[ \mathrm{e}^{-\alpha \tau_h^{+}} g \left( X_{\tau_h^{+}} \right)  \right] = g(x) \quad \left( \text{resp.} \quad \e_x \left[ \mathrm{e}^{-\alpha \tau_H^{-}} g \left( X_{\tau_H^{-}} \right) \right] = g(x) \right) ,
$$
and, if $x < h$ or $x > H$, then
$$
\e_x \left[ \mathrm{e}^{-\alpha \left( \tau_H^{+} \wedge \tau_h^{-} \right) } g \left( X_{\tau_H^{+} \wedge \tau_h^{-}} \right) \right] = g(x) .
$$

\section{Our main results}

Our first objective is to obtain the joint distribution of
$$
\left( \int_0^T \mathbf{1}_{\{h < X_t < H\}} \mathrm{d}t , X_T \right) ,
$$
for a given $T>0$. In order to do so, we will compute the following joint Laplace-Carson transform with respect to $T$: for each $x \in \reals$, set
\begin{equation}\label{E:defw}
w(x;h,H,\alpha,\rho,\gamma) := \int_0^\infty \alpha \mathrm{e}^{-\alpha T} \e_x \left[ \mathrm{e}^{- \rho \int_0^T \ind_{\{h < X_t < H\}} \mathrm{d}t + \gamma X_T} \right] \mathrm{d}T ,
\end{equation}
where $\alpha>0$, $\rho \geq 0$ and $\gamma \in \reals$. Clearly, we have
$$
w(x) = \e_x \left[ \mathrm{e}^{- \rho \int_0^{\exptime} \ind_{\{h < X_t < H\}} \mathrm{d}t + \gamma X_{\exptime}} \right] ,
$$
where $\exptime$ is an exponentially distributed random variable (independent of $X$) with mean $1/\alpha$.

Here is our main result.
\begin{theorem}\label{T:main}
For any $0\leq\gamma<\min(\eta_1,\theta_1), \rho>0$ and $G(\gamma)<\alpha$, we have
\begin{multline*}
\int_0^\infty \alpha \mathrm{e}^{-\alpha T} \e_x \left[ \mathrm{e}^{- \rho \int_0^T \ind_{\{h < X_t < H\}} \mathrm{d}t + \gamma X_T} \right] \mathrm{d}T \\
=
\begin{cases}
\sum_{i=1}^{m+1} \omega_{i}^{L} \mathrm{e}^{\beta_{i,\alpha} (x-h)} - c_L \mathrm{e}^{\gamma x} , & x\leq h ,\\
-\sum_{i=1}^{m+1} \omega_{i}^{0} \mathrm{e}^{\beta_{i,\alpha+\rho} (x-H)} - \sum_{j=1}^{n+1} \nu_{j}^{0} \mathrm{e}^{-\gamma_{j,\alpha+\rho} (x-h)} - c_{0} \mathrm{e}^{\gamma x} , & h<x< H ,\\
\sum_{j=1}^{n+1} \nu_j^{U} \mathrm{e}^{-\gamma_{j,\alpha} (x-H)} - c_U \mathrm{e}^{\gamma x} , & x \geq H ,
\end{cases} 
\end{multline*}
where
$$
c_L=c_U=\frac{\alpha}{G(\gamma)-\alpha}, \qquad c_0=\frac{\alpha}{G(\gamma)-(\alpha+\rho)} .
$$
The vector of coefficients
$$
Q=\left( \omega_{i}^{L}; \omega_i^0; \nu_j^0; \nu_j^{U}; i=1,\ldots,m+1; j=1,\ldots,n+1 \right)
$$
satisfies a linear system
\begin{equation}\label{E:linear_system}
BQ=V.
\end{equation}
Here $V$ is a $2S$-dimensional vector, 
\begin{equation}\label{V}
V=(c_U-c_0) \; \left( \begin{array}{c} V_1 \\ V_2 \end{array} \right)
\end{equation}
where the $S$-dimensional column vectors $V_1$ and $V_2$ are given by
\begin{align*}
V_1 &= \left( \mathrm{e}^{\gamma h} , \gamma \mathrm{e}^{\gamma h} , \frac{\mathrm{e}^{\gamma h}}{\eta_{1}-\gamma} , \cdots ,  \frac{\mathrm{e}^{\gamma h}}{\eta_{m}-\gamma} , \frac{\mathrm{e}^{\gamma h}}{\theta_{1}+\gamma} , \cdots , \frac{\mathrm{e}^{\gamma h}}{\theta_{n}+\gamma} \right) \\
V_2 &= \left( \mathrm{e}^{\gamma H} , \gamma \mathrm{e}^{\gamma H} , \frac{\mathrm{e}^{\gamma H}}{\eta_{1}-\gamma} , \cdots , \frac{\mathrm{e}^{\gamma H}}{\eta_{m}-\gamma} , \frac{\mathrm{e}^{\gamma H}}{\theta_{1}+\gamma} , \cdots , \frac{\mathrm{e}^{\gamma H}}{\theta_{n}+\gamma} \right)
\end{align*}
and $B$ is a $2S\times2S$ matrix
\begin{equation}\label{E:matrixB}
B = \left(
\begin{array}{cc}
M & N Z_{\beta}\\
M Z_{\gamma} & N
\end{array}
\right) ,
\end{equation}
where $Z_{\beta}$ and $Z_{\gamma}$ are $S \times S$ diagonal matrices with elements
$$
\left\lbrace 0,\ldots,0, \mathrm{e}^{\beta_{1,\alpha+\rho}(h-H)}, \ldots, \mathrm{e}^{\beta_{m+1,\alpha+\rho}(h-H)} \right\rbrace
$$
and
$$
\left\lbrace 0, \ldots, 0, \mathrm{e}^{\gamma_{1,\alpha+\rho}(h-H)}, \ldots, \mathrm{e}^{\gamma_{n+1,\alpha+\rho}(h-H)} \right\rbrace ,
$$
respectively, and where $M$ and $N$ are given by
\begin{equation*}
M=\left(
\begin{array}{cccccc}
1 & \cdots & 1 & 1 & \cdots & 1\\
\beta_{1,\alpha} & \cdots & \beta_{m+1,\alpha} & -\gamma_{1,\alpha+\rho} & \cdots & -\gamma_{n+1,\alpha+\rho}\\
\frac{1}{\eta_1-\beta_{1,\alpha}} & \cdots & \frac{1}{\eta_{1}-\beta_{m+1,\alpha}} & \frac{1}{\eta_{1}+\gamma_{1,\alpha+\rho}} & \cdots & \frac{1}{\eta_{1}+\gamma_{n+1,\alpha+\rho}}\\
\vdots & \ddots & \vdots &  \vdots & \ddots &\vdots \\
\frac{1}{\eta_m-\beta_{1,\alpha}} & \cdots & \frac{1}{\eta_m-\beta_{m+1,\alpha}} & \frac{1}{\eta_m+\gamma_{1,\alpha+\rho}} & \cdots & \frac{1}{\eta_m+\gamma_{n+1,\alpha+\rho}}\\
\frac{1}{\theta_{1}+\beta_{1,\alpha}} & \cdots & \frac{1}{\theta_{1}+\beta_{m+1,\alpha}} & \frac{1}{\theta_{1}-\gamma_{1,\alpha+\rho}} & \cdots & \frac{1}{\theta_{1}-\gamma_{n+1,\alpha+\rho}}\\
\vdots & \ddots & \vdots &  \vdots & \ddots &\vdots \\
\frac{1}{\theta_{n}+\beta_{1,\alpha}} & \cdots & \frac{1}{\theta_{n}+\beta_{m+1,\alpha}} & \frac{1}{\theta_{n}-\gamma_{1,\alpha+\rho}} & \cdots & \frac{1}{\theta_{n}-\gamma_{n+1,\alpha+\rho}}
\end{array}
\right)
\end{equation*}
and 
\begin{equation*}
N=\left(
\begin{array}{cccccc}
1 & \cdots & 1 & 1 & \cdots & 1\\
-\gamma_{1,\alpha} & \cdots & -\gamma_{n+1,\alpha} & \beta_{1,\alpha+\rho} & \cdots & \beta_{m+1,\alpha+\rho}\\
\frac{1}{\eta_1+\gamma_{1,\alpha}} & \cdots & \frac{1}{\eta_{1}+\gamma_{n+1,\alpha}} & \frac{1}{\eta_{1}-\beta_{1,\alpha+\rho}} & \cdots & \frac{1}{\eta_{1}-\beta_{m+1,\alpha+\rho}}\\
\vdots & \ddots & \vdots &  \vdots & \ddots &\vdots \\
\frac{1}{\eta_m+\gamma_{1,\alpha}} & \cdots & \frac{1}{\eta_{n}+\gamma_{n+1,\alpha}} & \frac{1}{\eta_m-\beta_{1,\alpha+\rho}} & \cdots & \frac{1}{\eta_m-\beta_{m+1,\alpha+\rho}}\\
\frac{1}{\theta_{1}-\gamma_{1,\alpha}} & \cdots & \frac{1}{\theta_{1}-\gamma_{n+1,\alpha}} & \frac{1}{\theta_{1}+\beta_{1,\alpha+\rho}} & \cdots & \frac{1}{\theta_{1}+\beta_{m+1,\alpha+\rho}}\\
\vdots & \ddots & \vdots &  \vdots & \ddots &\vdots \\
\frac{1}{\theta_{n}-\gamma_{1,\alpha}} & \cdots & \frac{1}{\theta_{n}-\gamma_{n+1,\alpha}} & \frac{1}{\theta_{n}+\beta_{1,\alpha+\rho}} & \cdots & \frac{1}{\theta_{n}+\beta_{m+1,\alpha+\rho}}
\end{array}
\right) .
\end{equation*}
\end{theorem}

In order for the last Theorem to yield an explicit result, we must show that the linear system in~\eqref{E:linear_system} is solvable:
\begin{lemma}\label{L:inversion}
Under Assumption~\ref{A:roots}, for a given value of $\alpha > 0$, the matrix $B$ given in~\eqref{E:matrixB} is invertible.
\end{lemma}
\begin{proof}
Let $S=m+n+2$ and assume that $BC=0$ for some vector $C=\left(C_1,C_2,\ldots,C_{2S}\right)$. Consider the function $V(x)=\sum_{i=1}^{2S}C_i \mathrm{e}^{\rho_i x}$ for $x\in(h,H)$, and $V(x)=0$ otherwise, with $\rho_{1},\ldots,\rho_{2S}$ be the distinct real zeros of the equation $G(x)=\alpha$. Since $BC=0$  and  $V(x)$ is a solution to the boundary value problem 
\begin{multline}\label{phi1}
\begin{cases}
\left( \mathcal{L}-\alpha-\rho \ind_{\{h<x < H\}} \right) \phi(x) = 0 & x \in (h,H) ,\\
\phi(x) = 0 & x \in (-\infty,h] \cup [H,+\infty) .
\end{cases} 
\end{multline}
From the uniqueness of the solution to the boundary value problem~\eqref{phi1}, $V(x) \equiv 0$ on $(h,H)$. Now, since $\{\mathrm{e}^{\rho_i x}, \leq i \leq 2S\}$ are linearly independent then $C=0$ and $B$ is invertible.
\end{proof}

Using the same methodology as in the proof of Theorem~\ref{T:main}, we can prove the following result.

\begin{theorem}\label{T:secondmain}
For any $0\leq\gamma<\min(\eta_1,\theta_1), \rho_1>0, \rho_2>0$ and $G(\gamma)<\alpha$, we have for $h<H$,
\begin{multline}
\int_0^\infty \alpha \mathrm{e}^{-\alpha T} \e_x \left[ \mathrm{e}^{- \rho_1 \int_0^T \ind_{\{X_t \leq h\}} \mathrm{d}t - \rho_2 \int_0^T \ind_{\{X_t \geq H\}} \mathrm{d}t + \gamma X_T} \right] \mathrm{d}T \\
=
\begin{cases}
\sum_{i=1}^{m+1} \omega_{i}^{L} \mathrm{e}^{\beta_{i,\alpha+\rho_1}(x-h)} - c_{L} \mathrm{e}^{\gamma x} , & x\leq h ,\\
- \sum_{i=1}^{m+1}\omega_{i}^{0} \mathrm{e}^{\beta_{i,\alpha}(x-H)} - \sum_{j=1}^{n+1} \nu_{j}^{0} \mathrm{e}^{-\gamma_{j,\alpha}(x-h)} - c_{0} \mathrm{e}^{\gamma x} , & h<x< H ,\\
\sum_{j=1}^{n+1} \nu_j^{U} \mathrm{e}^{-\gamma_{j,a+\rho_2}(x-H)} - c_{U} \mathrm{e}^{\gamma x}, &  x\geq H ,
\end{cases} 
\end{multline}
where
$$
c_L=\frac{\alpha}{G(\gamma)-(\alpha+\rho_1)}, \quad c_0=\frac{\alpha}{G(\gamma)-\alpha}, \quad \text{and} \quad c_U=\frac{\alpha}{G(\gamma)-(\alpha+\rho_2)} .
$$
The vector of coefficients 
$$
Q^\prime = \left( \omega_{i}^{L}; w_i^0; \nu_j^0; \nu_j^{U} ; i=1,\ldots,m+1; j=1,\ldots,n+1 \right)
$$
satisfies a linear system
$$
B^\prime Q^\prime = V .
$$  
Here $V$ is the vector defined in~\eqref{V}, $B^\prime$ is a $2S \times 2S$ matrix given by
\begin{equation*}
B^\prime = \left(
\begin{array}{cc}
M^\prime & N^\prime Z_{\beta}\\
M^\prime Z_{\gamma} & N^\prime
\end{array}
\right) ,
\end{equation*}
where $Z_{\beta}$ and $Z_{\gamma}$ are $S \times S$ diagonal matrices with elements
$$
\left\lbrace 0,\ldots,0, \mathrm{e}^{\beta_{1,\alpha}(h-H)},\ldots, \mathrm{e}^{\beta_{m+1,\alpha}(h-H)} \right\rbrace ,
$$
and
$$
\left\lbrace 0,\ldots,0, \mathrm{e}^{\gamma_{1,\alpha}(h-H)},\ldots, \mathrm{e}^{\gamma_{n+1,\alpha}(h-H)} \right\rbrace ,
$$
respectively, and where $M^\prime$ and $N^\prime$ are given by
\begin{equation*}
M^\prime = \left(
\begin{array}{cccccc}
1 & \cdots & 1 &1& \cdots & 1\\
\beta_{1,\alpha+\rho_1} & \cdots & \beta_{m+1,\alpha+\rho_1} & -\gamma_{1,\alpha} & \cdots & -\gamma_{n+1,\alpha}\\
\frac{1}{\eta_1-\beta_{1,\alpha+\rho_1}} & \cdots & \frac{1}{\eta_{1}-\beta_{m+1,\alpha+\rho_1}} & \frac{1}{\eta_{1}+\gamma_{1,\alpha}} & \cdots & \frac{1}{\eta_{1}+\gamma_{n+1,\alpha}}\\
\vdots & \ddots & \vdots &  \vdots & \ddots &\vdots \\
\frac{1}{\eta_m-\beta_{1,\alpha+\rho_1}} & \cdots & \frac{1}{\eta_m-\beta_{m+1,\alpha+\rho_1}} & \frac{1}{\eta_m+\gamma_{1,\alpha}} & \cdots & \frac{1}{\eta_m+\gamma_{n+1,\alpha}}\\
\frac{1}{\theta_{1}+\beta_{1,\alpha+\rho_1}} & \cdots & \frac{1}{\theta_{1}+\beta_{m+1,\alpha+\rho_1}} & \frac{1}{\theta_{1}-\gamma_{1,\alpha}} & \cdots & \frac{1}{\theta_{1}-\gamma_{n+1,\alpha}}\\
\vdots & \ddots & \vdots &  \vdots & \ddots &\vdots \\
\frac{1}{\theta_{n}+\beta_{1,\alpha+\rho_1}} & \cdots & \frac{1}{\theta_{n}+\beta_{m+1,\alpha+\rho_1}} & \frac{1}{\theta_{n}-\gamma_{1,\alpha}} & \cdots & \frac{1}{\theta_{n}-\gamma_{n+1,\alpha}}
\end{array}
\right)
\end{equation*}
and
\begin{equation*}
N^\prime = \left(
\begin{array}{cccccc}
1 & \cdots & 1 & 1 & \cdots & 1\\
-\gamma_{1,\alpha+\rho_2} & \cdots & -\gamma_{n+1,\alpha+\rho_2} & \beta_{1,\alpha} & \cdots & \beta_{m+1,\alpha}\\
\frac{1}{\eta_1+\gamma_{1,\alpha+\rho_2}} & \cdots & \frac{1}{\eta_{1}+\gamma_{n+1,\alpha+\rho_2}} & \frac{1}{\eta_{1}-\beta_{1,\alpha}} & \cdots & \frac{1}{\eta_{1}-\beta_{m+1,\alpha}}\\
\vdots & \ddots & \vdots &  \vdots & \ddots &\vdots \\
\frac{1}{\eta_{m}+\gamma_{1,\alpha+\rho_2}} & \cdots & \frac{1}{\eta_m+\gamma_{n+1,\alpha+\rho_2}} & \frac{1}{\eta_m-\beta_{1,\alpha}} & \cdots & \frac{1}{\eta_{m}-\beta_{m+1,\alpha}}\\
\frac{1}{\theta_{1}-\gamma_{1,\alpha+\rho_2}} & \cdots & \frac{1}{\theta_{1}-\gamma_{n+1,\alpha+\rho_2}} & \frac{1}{\theta_{1}+\beta_{1,\alpha}} & \cdots & \frac{1}{\theta_{1}+\beta_{m+1,\alpha}}\\
\vdots & \ddots & \vdots &  \vdots & \ddots &\vdots \\
\frac{1}{\theta_{n}-\gamma_{1,\alpha+\rho_2}} & \cdots & \frac{1}{\theta_{n}-\gamma_{n+1,\alpha+\rho_2}} & \frac{1}{\theta_{n}+\beta_{1,\alpha}} & \cdots & \frac{1}{\theta_{n}+\beta_{m+1,\alpha}}
\end{array}
\right) .
\end{equation*}
\end{theorem}

In Theorem~\ref{T:secondmain}, if we let $H \to h$, then it greatly simplifies the expression:
\begin{corollary}
Given the constants $\rho_1>0, \rho_2>0, \gamma \geq 0$ and $ \alpha>0$ such that $\gamma < \min \left( \eta_1, \theta_1 \right)$ and $G(\gamma) < \alpha$, we have
%
%
\begin{multline*}
\int_0^\infty \alpha \mathrm{e}^{-\alpha T} \e_x \left[ \mathrm{e}^{-\rho_1 \int_{0}^{T} \ind_{\{X_s \leq h\} \mathrm{d}s} - \rho_2 \int_{0}^{T} \ind_{\{X_s \geq h\}} + \gamma X_{T}} \right] \mathrm{d}T \\
=
\begin{cases}
\sum_{i=1}^{m+1} \omega_{i} \mathrm{e}^{\beta_{i,\alpha+\rho_1}(x-h)} - c_{1} \mathrm{e}^{\gamma x} , & x \leq h ,\\
\sum_{j=1}^{n+1} \nu_j \mathrm{e}^{-\gamma_{j,\alpha+\rho_2}(x-h)} - c_{2} \mathrm{e}^{\gamma x}, & x > h ,
\end{cases} 
\end{multline*}
where
$$
c_1 = \frac{\alpha}{G(\gamma)-(\alpha + \rho_1)} , \quad c_{2} = \frac{\alpha}{G(\gamma)-(\alpha+\rho_2)} . 
$$
For $i=1,\ldots,m+1$,
\begin{eqnarray*}
&&\omega_{i}=\frac{\prod_{j=1,j\neq i}^{m+1}(\beta_{j,\alpha+\rho_1}-\gamma)\prod_{k=1}^{n+1}(-\gamma_{k,\alpha+\rho_2}-\gamma)\prod_{j=1}^{m}(\eta_j-\beta_{i,\alpha+\rho_1})\prod_{k=1}^{n}(\theta_k+\beta_{i,\alpha+\rho_1})}{\prod_{j=1,j\neq i}^{m+1}(\beta_{j,\alpha+\rho_1}-\beta_{i,\alpha+\rho_1})\prod_{k=1}^{n+1}(-\gamma_{k,\alpha+\rho_2}-\beta_{i,\alpha+\rho_1})\prod_{j=1}^{m}(\eta_j-\gamma)\prod_{k=1}^{n}(\theta_k+\gamma)} c_{12}, 
\end{eqnarray*}
and, for $i=1,\ldots,n+1$,
\begin{eqnarray*}
&&\nu_i=\frac{\prod_{j=1}^{m+1}(\beta_{j,\alpha+\rho_1}-\gamma)\prod_{k=1,k\neq i}^{n+1}(-\gamma_{k,\alpha+\rho_2}-\gamma)\prod_{j=1}^{m}(\eta_j+\gamma_{i,\alpha+\rho_2})\prod_{k=1}^{n}(\theta_k-\gamma_{i,\alpha+\rho_2})}{\prod_{j=1}^{m+1}(\beta_{j,\alpha+\rho_1}+\gamma_{i,\alpha+\rho_2})\prod_{k=1,k\neq i}^{n+1}(-\gamma_{k,\alpha+\rho_2}+\gamma_{i,\alpha+\rho_2})\prod_{j=1}^{m}(\eta_j-\gamma)\prod_{k=1}^{n}(\theta_k+\gamma)} c_{12},
\end{eqnarray*}
with $c_{12}=c_1-c_2$.
\end{corollary}
\begin{proof}
By Gauss elimination, we can show that the determinant of
\begin{equation*}
A=\left(\begin{array}{cccc}
1 & 1 &  \cdots & 1\\
a_{1} & a_{2}& \cdots & a_{S}\\
\frac{1}{\eta_1-a_1} & \frac{1}{\eta_{1}-a_2}  & \cdots & \frac{1}{\eta_{1}-a_{S}}\\
\vdots &  \vdots &  \ddots &\vdots \\
\frac{1}{\eta_{m}-a_1} &  \frac{1}{\eta_{m}-a_2} &  \cdots & \frac{1}{\eta_{m}-a_{S}}\\
\frac{1}{\theta_{1}+a_1} & \frac{1}{\theta_{1}+a_2}  & \cdots & \frac{1}{\theta_{1}+a_{S}}\\
\vdots &  \vdots  & \ddots &\vdots \\
\frac{1}{\theta_{n}+a_1} & \frac{1}{\theta_{n}+a_2}  & \cdots & \frac{1}{\theta_{n}+a_{S}}
\end{array}\right)
\end{equation*}
where $S=m+n+2$, is given by
$$
\det(A)=-\frac{(\sum_{i=1}^{m}\eta_i+\sum_{j=1}^{n}\theta_j)\prod_{1\leq i<j\leq S}(a_i-a_j)}{\prod_{1\leq i\leq S, 1\leq j\leq S}(\prod_{1\leq k\leq m}(\eta_k-a_i)\prod_{1\leq l\leq n}(\theta_l+a_j))} .
$$
If $\det(A) \neq 0$, then the matrix $A$ is invertible and, for the column vector
$$
\Delta= \left( 1, \gamma, \frac{1}{\eta_{1}-\gamma},\ldots,\frac{1}{\eta_{m}-\gamma}, \frac{1}{\theta_{1}+\gamma},\ldots, \frac{1}{\theta_{n}+\gamma} \right) ,
$$
the linear system $A Y = \Delta$ has a unique solution $Y^{*}= \left( y_{1}^{*},\ldots,y_{S}^{*} \right)$, where, for $i = 1,2,\dots, S$,
$$
y_{i}^{*} = \frac{\prod_{j\neq i}(a_j-\gamma)\prod_{k=1}^{m}(\eta_k-a_i)\prod_{l=1}^{n}(\theta_k+a_i)}{\prod_{j\neq i}(a_j-a_i)\prod_{k=1}^{m}(\eta_k-\gamma)\prod_{l=1}^{n}(\theta_k+\gamma)} .
$$ 
Applying the above to Theorem~\ref{T:secondmain} when $H \to h$ completes the proof.
\end{proof}

\section{Occupation-time option pricing}

We now show how our theoretical results can be easily applied to the pricing of various occupation-time options. The idea is to obtained explicit expressions for (double) Laplace transforms of option prices, which can then be inverted using well-known and well-studied Laplace inversion techniques to get numerical prices; see e.g.\ \cite{kou_et_al_2005} and \cite{cai_kou_2011} and the references therein. Note that using this methodology, together with the results of Section 3, many other (and more complicated) occupation-time derivatives could be analyzed.

\subsection{Step and double-barrier step options}

As already mentioned in the Introduction, a (down-and-out call) step option admits the following payoff:
$$
\mathrm{e}^{- \rho \int_0^T \mathbf{1}_{\{S_t \leq L\}} \mathrm{d}t } \left( S_T - K \right)^+ = \mathrm{e}^{- \rho \int_0^T \mathbf{1}_{\{X_t \leq \ln(L/S_0)\}} \mathrm{d}t } \left( S_0 \mathrm{e}^{X_T} - K \right)^+ .
$$
Then, its price can be written as
$$
C^{\text{step}} (k,T) := \mathrm{e}^{-r T} \e \left[ \mathrm{e}^{- \rho \int_0^T \mathbf{1}_{\{X_t \leq \ln(L/S_0)\}} \mathrm{d}t } \left( S_0 \mathrm{e}^{X_T} - \mathrm{e}^{-k} \right)^+ \right] ,
$$
where $k=-\ln(K)$. Following Carr and Madan's approach for vanilla options, as in \cite{carr_madan_1999} (see also \cite{kou_et_al_2005} and \cite{cai_kou_2011}), we can easily compute the double Laplace transform of $C^{\text{step}} (k,T)$:
$$
\int_0^\infty \int_{-\infty}^\infty \mathrm{e}^{- \alpha T - \beta k} C^{\text{step}} (k,T) \mathrm{d}k \mathrm{d}T = \mathrm{e}^{-r T} \frac{S_0^{\beta+1} w \left( \ln(S_0) ; -\infty, \ln(L/S_0), \alpha, \rho, \beta+1 \right)}{\alpha \beta (\beta+1)} ,
$$
where $w(x;h,H,\alpha,\rho,\gamma)$ is given by Theorem~\ref{T:main}. Note that the double Laplace transform of delta of the latter option can be easily obtained from the above.

Recall that the payoff of a double-barrier step option is given by
$$
\mathrm{e}^{- \rho^- \int_0^T \mathbf{1}_{\{S_t \leq L\}} \mathrm{d}t - \rho^+ \int_0^T \mathbf{1}_{\{S_t \geq U\}} \mathrm{d}t} \left( S_T - K \right)_+ ,
$$
where $\rho^-$ and $\rho^+$ are the knock-out rates. Then, its price can be written as
$$
C^{\text{double}} (k,T) := \mathrm{e}^{-r T} \e \left[ \mathrm{e}^{- \rho^- \int_0^T \mathbf{1}_{\{X_t \leq \ln(L/S_0)\}} \mathrm{d}t - \rho^+ \int_0^T \mathbf{1}_{\{X_t \geq \ln(U/S_0)\}} \mathrm{d}t} \left( S_0 \mathrm{e}^{X_T} - \mathrm{e}^{-k} \right)^+ \right] ,
$$
where $k=-\ln(K)$. Again, we can easily compute its double Laplace transform:
$$
\int_0^\infty \int_{-\infty}^\infty \mathrm{e}^{- \alpha T - \beta k} C^{\text{double}} (k,T) \mathrm{d}k \mathrm{d}T = \mathrm{e}^{-r T} \frac{S_0^{\beta+1} w \left( \ln(S_0) ; \ln(L/S_0), \ln(U/S_0)\alpha, \rho^-, \rho^+, \beta+1 \right)}{\alpha \beta (\beta+1)} ,
$$
where an explicit expression for
$$
w(x ; h, H, \alpha, \rho^-, \rho^+, \gamma) = \int_0^\infty \alpha \mathrm{e}^{-\alpha T} \e_x \left[ \mathrm{e}^{- \rho_1 \int_0^T \ind_{\{X_t \leq h\}} \mathrm{d}t - \rho_2 \int_0^T \ind_{\{X_t \geq H\}} \mathrm{d}t + \gamma X_T} \right] \mathrm{d}T
$$
is given by Theorem~\ref{T:secondmain}.


\subsection{Quantile options}

Recall from the Introduction that a fixed-strike $\alpha$-quantile call option admits the following payoff: for $0 \leq \alpha \leq 1$,
$$
\left( S_0 \mathrm{e}^{\gamma q(\alpha,T)}  - K \right)^+ ,
$$
where
$$
q(\alpha,T) = \inf \left\lbrace h \colon \int_0^T \mathbf{1}_{\{X_t \leq h\}} \mathrm{d}t > \alpha T \right\rbrace .
$$
For any $0 \leq \upsilon \leq T$, the price of this $\alpha$-quantile call option can be written as
$$
C^{\text{quantile}} (\upsilon,T) = \mathrm{e}^{-rT} \e \left[ (S_0 \mathrm{e}^{\lambda q(\upsilon/T,T)} - K )_+ \right] .
$$
Then, the double Laplace transform of $C^{\text{quantile}} (\upsilon,T)$ is given by 
\begin{multline*}
\int_{0}^{\infty}\int_{0}^{\infty} \mathrm{e}^{- \alpha T - \rho \upsilon} C^{\text{quantile}} (\upsilon,T) \ind_{\{\upsilon <T\}} \mathrm{d}T \mathrm{d}\upsilon \\
=
\begin{cases}
\sum_{i=1}^{m+1} \frac{\lambda K}{\rho}\frac{\omega_i}{\beta_{i,\alpha+\rho} - \lambda}(S_0/K)^{(\beta_{i,\alpha+\rho}+\lambda)/\lambda} & \text{if $K \leq S_0$,}\\
\sum_{i=1}^{m+1} \frac{\lambda K}{\rho}\frac{\omega_i}{\beta_{i,\alpha+\rho} - \lambda}(S_0/K)^{(\beta_{i,\alpha+\rho}+\lambda)/\lambda} & \\
\quad - \sum_{j=1}^{n+1} \frac{\lambda K}{\rho}\frac{\nu_j}{\gamma_{j,\alpha}+\lambda} (1-(S_0/K)^{(\gamma_{i,\alpha+\rho}+\lambda)/\lambda}) + \frac{(S_{0}-K)}{\alpha(\alpha+\rho)} & \text{if $K < S_0$,}
\end{cases} 
\end{multline*}
where $\{\omega_i, i=1,\ldots,m+1\}$ and $\{\nu_j, j=1,\ldots,n+1\}$ are given by Theorem~\ref{T:main}.

\section{Acknowledgements}
Funding in support of this work was provided by an ENC grant from the Fonds de recherche du Qu\'ebec - Nature et technologies (FRQNT). J.-F. Renaud also thanks the Insti\-tut de finance math\'ematique de Montr\'eal (IFM2) for its financial support.

\section{Proof of Theorem~\ref{T:main}}

We have for $x<h$
\begin{eqnarray}
\nonumber w(x) &=& \e_x \left[ \mathrm{e}^{-\rho\int_{0}^{\exptime} \ind_{\{h<X_s< H\}} \mathrm{d}s + \gamma X_{\exptime}} \right] \\
\nonumber&=& \e_x \left[ \mathrm{e}^{-\rho \int_{0}^{\exptime} \ind_{\{h<X_s< H\}} \mathrm{d}s + \gamma X_{\exptime}}, \exptime < \tau_h^+ \right] + \e_x \left[ \mathrm{e}^{-\rho \int_{0}^{\exptime} \ind_{\{h<X_s\leq H\}} \mathrm{d}s + \gamma X_{\exptime}} , \tau_h^+ \leq \exptime \right] \\
\nonumber &=& \e_x \left[ \mathrm{e}^{\gamma X_{\exptime}}, \exptime < \tau_h^+ \right] + \e_x \left[ \mathrm{e}^{-\rho \int_{0}^{\exptime} \ind_{\{h<X_s \leq H\}} \mathrm{d}s + \gamma X_{\exptime}} , \tau_h^+ \leq \exptime \right] \\
\label{mark} &=& \e_x \left[ \int_0^{\tau_h^+} \alpha \mathrm{e}^{-\alpha s +\gamma X_s} \mathrm{d}s \right] + \e_x \left[ \mathrm{e}^{-\rho \int_{0}^{\exptime} \ind_{\{h< X_s \leq H\}} \mathrm{d}s + \gamma X_{\exptime}} , \tau_h^+ \leq \exptime \right] .
\end{eqnarray}

Applying It\^o's formula to the process $\{\mathrm{e}^{-\alpha t +\gamma X_t}, t \geq 0\}$, we obtain that the process
\begin{eqnarray*}
M_t &:=& \mathrm{e}^{-\alpha(t\wedge \tau^{+}_h) + \gamma X_{t\wedge \tau^{+}_h}} - \mathrm{e}^{\gamma X_0} - \int_{0}^{t\wedge \tau^{+}_h} \mathrm{e}^{-\alpha s}\left(-\alpha \mathrm{e}^{\gamma X_s}+\mathcal{L} \mathrm{e}^{\gamma X_s}\right) \mathrm{d}s\\
&=& \mathrm{e}^{-\alpha(t\wedge \tau^{+}_h)+\gamma X_{t\wedge \tau^{+}_h}} - \mathrm{e}^{\gamma x}-(G(\gamma)-\alpha) \int_{0}^{t\wedge \tau^{+}_h} \mathrm{e}^{-\alpha s+\gamma X_s} \mathrm{d}s,
\end{eqnarray*}
is a local martingale starting from $M_0=0$. Since $G(\gamma)<\alpha$, it follows from Fubini's theorem that
$$
\e\left[\int_{0}^{t} \mathrm{e}^{-\alpha s+\gamma X_s} \mathrm{d}s \right] = \int_{0}^{t} \mathrm{e}^{-\alpha s} \e \left[ \mathrm{e}^{\gamma X_s} \right] \mathrm{d}s = \int_{0}^{t} \mathrm{e}^{(-\alpha+G(\gamma)) s} \mathrm{d}s = \frac{\mathrm{e}^{(-\alpha+G(\gamma))t}-1}{(-\alpha+G(\gamma))} < \infty ,
$$
for all $t \geq 0$. So, using Lebesgue's dominated convergence theorem, we have that $\{M_t , t\geq 0\}$ is actually a martingale. In particular,
\begin{equation}\label{covth}
\e_x \left[ \mathrm{e}^{-\alpha\tau_h^{+}+\gamma X_{\tau_{h}^{+}}} - \mathrm{e}^{\gamma x}\right] = \left( G(\gamma)-\alpha \right) \e_x \left[ \int_{0}^{ \tau^{+}_h} \mathrm{e}^{-\alpha s+\gamma X_s} \mathrm{d}s \right] .
\end{equation}
Plugging~\eqref{covth} into~\eqref{mark}, we get, by the strong Markov property of $X$ and the lack-of-memory property of $\exptime$, that
\begin{equation}\label{gamma1}
w(x) = \frac{\alpha}{G(\gamma)-\alpha} \left( \e_x \left[ \mathrm{e}^{-\alpha\tau_h^{+}+ X_{\tau_{h}^{+}}} \right] - \mathrm{e}^{\gamma x} \right) + \e_x \left[ \mathrm{e}^{-\alpha\tau_h^{+}} w\left( X_{\tau_{h}^{+}} \right) \right] .
\end{equation}
Similarly, for $x>H$,
\begin{equation}\label{gamma2}
w(x) = \frac{\alpha}{G(\gamma)-\alpha} \left( \e_x \left[ \mathrm{e}^{-\alpha\tau_h^{-}+\gamma X_{\tau_{h}^{-}}} \right] - \mathrm{e}^{\gamma x}\right) + \e_x \left[ \mathrm{e}^{-\alpha\tau_h^{-}} w\left( X_{\tau_{h}^{-}} \right)  \right] ,
\end{equation}
and, for $h \leq x\leq H$,
\begin{equation}\label{gamma3}
w(x) = \frac{\alpha}{G(\gamma)-(\alpha+\rho)} \left( \e_x \left[ \mathrm{e}^{-(\alpha+\rho)\tau+\gamma X_{\tau}} \right] - \mathrm{e}^{\gamma x}\right) + \e_x \left[ \mathrm{e}^{-(\rho+\alpha)\tau} w \left( X_{\tau} \right) \right] .
\end{equation}

Define the following:
\begin{equation*}
w(x) =
\begin{cases}
w_1 (x) , & x \leq h ,\\
w_2 (x) , & h<x< H ,\\
w_3 (x) , & x \geq H ,
\end{cases} 
\end{equation*}
Combining Theorem~\ref{th1} with equations~\eqref{gamma1}, \eqref{gamma2} and \eqref{gamma3}, we get that $w(x)$ must be of the following form:
\begin{equation}\label{E:thews}
\e_x \left[ \mathrm{e}^{-\rho \int_{0}^{\exptime} \ind_{\{h < X_t < H\}} \mathrm{d}t + \gamma X_{\exptime}} \right] =
\begin{cases}
\sum_{i=1}^{m+1} \omega_{i}^{L} \mathrm{e}^{\beta_{i,\alpha} (x-h)} - c_L \mathrm{e}^{\gamma x} , & x\leq h ,\\
-\sum_{i=1}^{m+1} \omega_{i}^{0} \mathrm{e}^{\beta_{i,\alpha+\rho} (x-H)}\\
 \hspace{1cm} - \sum_{j=1}^{n+1} \nu_{j}^{0} \mathrm{e}^{-\gamma_{j,\alpha+\rho} (x-h)} - c_{0} \mathrm{e}^{\gamma x} , & h<x< H ,\\
\sum_{j=1}^{n+1} \nu_j^{U} \mathrm{e}^{-\gamma_{j,\alpha} (x-H)} - c_U \mathrm{e}^{\gamma x} , & x \geq H ,
\end{cases} 
\end{equation}
with $\omega^{L}_i, \omega^{0}_i, \nu^{0}_j$ and $ \nu^{U}_j $ to be determined. Now, we need equations to determine these coefficients.

Using again equations~\eqref{gamma1}, \eqref{gamma2} and \eqref{gamma3}, we have that $w(x)$ must satisfy
\begin{equation}\label{pb}
\left(\mathcal{L}-\alpha-\rho \ind_{\{h<x < H\}}\right) w(x) = - \alpha \mathrm{e}^{\gamma x} , \quad x \in \mathbb{R} \setminus \{h,H\} .
\end{equation}
Then, equation~\eqref{pb} can be rewritten as three separate equations in the regions $(-\infty,h), (h,H)$ and $(H,+\infty)$.

For $x<h$,
\begin{multline}\label{W1} 
- \alpha \mathrm{e}^{\gamma x} = \frac{\sigma^{2}}{2} w_{1}^{\prime \prime}(x) + \mu w_{1}^{\prime}(x) - (\lambda+\alpha) w_{1}(x) \\
+ \lambda \left\lbrace \int_{-\infty}^{0} w_1(x+y) \sum_{j=1}^{n} q_j \theta_j \mathrm{e}^{\theta_j y} \mathrm{d}y + \int_{0}^{h-x} w_1(x+y) \sum_{j=1}^{m} p_j \eta_j \mathrm{e}^{-\eta_j y} \mathrm{d}y \right. \\
+ \left. \int_{h-x}^{H-x} w_2(x+y) \sum_{j=1}^{m} p_j \eta_j \mathrm{e}^{-\eta_j y} \mathrm{d}y + \int_{H-x}^{+\infty} w_3(x+y) \sum_{j=1}^{m} p_j \eta_j \mathrm{e}^{-\eta_j y} \mathrm{d}y \right\rbrace .
\end{multline}
For $h<x<H$,
\begin{multline}\label{W2}
- \alpha \mathrm{e}^{\gamma x} = \frac{\sigma^{2}}{2}w_{2}^{\prime \prime}(x) + \mu w_{2}^{\prime}(x) - (\lambda+\rho+\alpha) w_{2}(x) \\
+ \lambda \left\lbrace \int_{-\infty}^{h-x} w_1(x+y) \sum_{j=1}^{n} q_j \theta_j \mathrm{e}^{\theta_j y} \mathrm{d}y + \int_{h-x}^{0} w_2(x+y) \sum_{j=1}^{n} q_i \theta_i \mathrm{e}^{\theta_i y} \mathrm{d}y \right. \\
+ \left. \int_{0}^{H-x} w_2(x+y) \sum_{j=1}^{m} p_j \eta_j \mathrm{e}^{-\eta_j y} \mathrm{d}y +\int_{H-x}^{+\infty} w_3(x+y) \sum_{j=1}^{m} p_j \eta_j \mathrm{e}^{-\eta_j y} \mathrm{d}y \right\rbrace .
\end{multline}
And finally, for  $x>H$,
\begin{multline}\label{W3}
- \alpha \mathrm{e}^{\gamma x}  = \frac{\sigma^{2}}{2} w_{3}^{\prime \prime}(x) + \mu w_{3}^{\prime}(x) - (\lambda+\alpha) w_{3}(x) \\
+ \lambda \left\lbrace \int_{-\infty}^{h-x} w_1(x+y) \sum_{j=1}^{n} q_j \theta_j \mathrm{e}^{\theta_j y} \mathrm{d}y +\int_{h-x}^{H-x} w_2(x+y) \sum_{j=1}^{n} q_i \theta_i \mathrm{e}^{\theta_i y} \mathrm{d}y \right. \\
+ \left. \int_{H-x}^{0} w_3(x+y) \sum_{j=1}^{n} q_i \theta_i \mathrm{e}^{\theta_i y} \mathrm{d}y + \int_{0}^{+\infty} w_3(x+y) \sum_{j=1}^{m} p_j \eta_j \mathrm{e}^{-\eta_j y} \mathrm{d}y \right\rbrace .
\end{multline}

Substituting the expression obtained in~\eqref{E:thews} into~\eqref{W1}, \eqref{W2} and \eqref{W3}, we get, for $x<h$,
\begin{multline*}
0 = \sum_{j=1}^{m} p_j \eta_j \mathrm{e}^{\eta_j (x-h)} \left\lbrace
\sum_{i=1}^{m+1} \left( \frac{\omega_{i}^{L}}{\eta_j-\beta_{i,\alpha}}+\frac{\omega_{i}^{0} \mathrm{e}^{\beta_{i,\alpha+\rho}(h-H)}}{\eta_j-\beta_{i,\alpha+\rho}} \right) + \sum_{i=1}^{n+1} \left( \frac{\nu_{i}^{0}}{\eta_j+\gamma_{i,\alpha+\rho}} \right) - (c_L-c_0) \frac{\mathrm{e}^{\gamma h}}{\eta_j-\gamma}
\right\rbrace \\
+ \sum_{j=1}^{m} p_j \eta_j \mathrm{e}^{\eta_j (x-H)} \left\lbrace 
\sum_{i=1}^{m+1} \left( \frac{\omega_{i}^{0}}{\eta_j-\beta_{i,\alpha+\rho}} \right) + \sum_{i=1}^{n+1} \left( \frac{\nu_{i}^{0} \mathrm{e}^{-\gamma_{i,\alpha+\rho}(h-H)}}{\eta_j+\gamma_{i,\alpha+\rho}} + \frac{\nu_{i}^{U}}{\eta_j+\gamma_{i,\alpha}} \right) - (c_L-c_0) \frac{\mathrm{e}^{\gamma H}}{\eta_j-\gamma} 
\right\rbrace .
\end{multline*}
and, for $x>H$,
\begin{multline*}
0 = \sum_{j=1}^{n} q_j \theta_j \mathrm{e}^{\theta_j (h-x)} \left\lbrace 
\sum_{i=1}^{m+1} \left( \frac{\omega_{i}^{L}}{\theta_j+\beta_{i,\alpha}} + \frac{\omega_{i}^{0} \mathrm{e}^{\beta_{i,\alpha+\rho}(h-H)}}{\theta_j+\beta_{i,\alpha+\rho}} \right) + \sum_{i=1}^{n+1} \left( \frac{\nu_{i}^{0}}{\theta_j-\gamma_{i,\alpha+\rho}} \right) - (c_L -c_0) \frac{\mathrm{e}^{\gamma h}}{\theta_j+\gamma} \right\rbrace \\
+ \sum_{j=1}^{n} q_j \theta_j \mathrm{e}^{\theta_j (H-x)} \left\lbrace 
\sum_{i=1}^{m+1} \left( \frac{\omega_{i}^{0}}{\theta_j+\beta_{i,\alpha+\rho}} \right) +\sum_{i=1}^{n+1} \left( \frac{-\nu_{i}^{0} \mathrm{e}^{-\gamma_{i,\alpha+\rho}(h-H)}}{\theta_j-\gamma_{i,\alpha+\rho}} + \frac{\nu_{i}^{U}}{\theta_j-\gamma_{i,\alpha}} \right) - (c_L -c_0) \frac{\mathrm{e}^{\gamma H}}{\theta_j+\gamma} \right\rbrace .
\end{multline*}

Therefore, the vector $Q$ or, in other words, the coefficients $\{\omega_i^{L}, i=1,\dots,m+1\}$, $\{\omega_i^{0}, i=1,\dots,m+1\}$, $\{\nu_i^{0}, i=1,\dots,n+1\}$ and $\{\nu_i^{U}, i=1,\dots,n+1\}$ satisfy the following: for each $j=1,\ldots,m$,
\begin{align*}
0 &= \sum_{i=1}^{m+1} \left( \frac{\omega_{i}^{L}}{\eta_j-\beta_{i,\alpha}} + \frac{\omega_{i}^{0} \mathrm{e}^{\beta_{i,\alpha+\rho}(h-H)}}{\eta_j-\beta_{i,\alpha+\rho}} \right) + \sum_{i=1}^{n+1} \left( \frac{\nu_{i}^{0}}{\eta_j+\gamma_{i,\alpha+\rho}} \right) - (c_L-c_0) \frac{\mathrm{e}^{\gamma h}}{\eta_j-\gamma} ,\\
0 &= \sum_{i=1}^{m+1} \left( \frac{\omega_{i}^{0}}{\eta_j-\beta_{i,\alpha+\rho}} \right) + \sum_{i=1}^{n+1} \left( \frac{-\nu_{i}^{0} \mathrm{e}^{-\gamma_{i,\alpha+\rho}(h-H)}}{\eta_j+\gamma_{i,\alpha+\rho}} + \frac{\nu_{i}^{U}}{\eta_j+\gamma_{i,\alpha}} \right) - (c_L-c_0) \frac{\mathrm{e}^{\gamma H}}{\eta_j-\gamma} ,
\end{align*}
and, for each $j=1,\ldots,n$,
\begin{align*}
0 &= \sum_{i=1}^{m+1} \left( \frac{\omega_{i}^{L}}{\theta_j+\beta_{i,\alpha}} + \frac{\omega_{i}^{0} \mathrm{e}^{\beta_{i,\alpha+\rho}(h-H)}}{\theta_j+\beta_{i,\alpha+\rho}} \right) + \sum_{i=1}^{n+1} \left( \frac{\nu_{i}^{0}}{\theta_j-\gamma_{i,\alpha+\rho}} \right) - (c_L-c_0) \frac{\mathrm{e}^{\gamma h}}{\theta_j+\gamma} ,\\
0 &= \sum_{i=1}^{m+1} \left( \frac{\omega_{i}^{0}}{\theta_j+\beta_{i,\alpha+\rho}} \right) + \sum_{i=1}^{n+1} \left( \frac{-\nu_{i}^{0} \mathrm{e}^{-\gamma_{i,\alpha+\rho}(h-H)}}{\theta_j-\gamma_{i,\alpha+\rho}} + \frac{\nu_{i}^{U}}{\theta_j-\gamma_{i,\alpha}} \right) - (c_L-c_0) \frac{\mathrm{e}^{\gamma H}}{\theta_j+\gamma} .
\end{align*}
In addition, we also have the following four equations:
\begin{eqnarray}
\label{cont1} && \sum_{i=1}^{m+1} \omega^{L}_i - c_L \mathrm{e}^{\gamma h} = \sum_{i=1}^{m+1} - \omega^{0}_i \mathrm{e}^{\beta_{i,\rho+\alpha}(h-H)} - \sum_{i=1}^{n+1} \nu_{i}^{0} - c_0 \mathrm{e}^{\gamma h} ,\\
\label{cont2} && \sum_{i=1}^{n+1} \nu^{U}_i - c_U \mathrm{e}^{\gamma H} = \sum_{i=1}^{m+1} - \omega^{0}_i - \sum_{i=1}^{n+1} \nu_{i}^{0} \mathrm{e}^{-\gamma_{i,\rho+\alpha}(H-h)} - c_0 \mathrm{e}^{\gamma H} ,\\
\label{dif1} && \sum_{i=1}^{m+1} \omega^{L}_i \beta_{i,\alpha} - c_L \gamma \mathrm{e}^{\gamma h} = \sum_{i=1}^{m+1} - \omega^{0}_i \beta_{i,\alpha+\rho} \mathrm{e}^{\beta_{i,\rho+\alpha}(h-H)} + \sum_{i=1}^{n+1} \nu_{i}^{0} \gamma_{i,\alpha+\rho} - c_0 \gamma \mathrm{e}^{\gamma h} ,\\
\label{dif2} && \sum_{i=1}^{n+1} - \nu^{U}_i \gamma_{i,\alpha} - c_U \gamma \mathrm{e}^{\gamma H} = \sum_{i=1}^{m+1} - \omega^{0}_i \beta_{i,\alpha+\rho} + \sum_{i=1}^{n+1} \nu_{i}^{0} \gamma_{i,\alpha+\rho} \mathrm{e}^{-\gamma_{i,\rho+\alpha}(H-h)} - c_0 \gamma \mathrm{e}^{\gamma H} .
\end{eqnarray}
Indeed, equations~\eqref{cont1} and \eqref{cont2} are immediate from the fact that $w(x)$ is continuous at $x=h$ and $x=H$.

For the proofs of equations~\eqref{dif1} and \eqref{dif2}, note that we have, for $h<x<H$,
\begin{multline*}
- \alpha \mathrm{e}^{\gamma x} = \frac{\sigma^{2}}{2} w_{2}^{\prime \prime}(x) + \mu w_{2}^{\prime}(x) - (\lambda+\rho+\alpha) w_{2}(x) \\
+ \lambda \left\lbrace \int_{-\infty}^{h-x} w_1(x+y) \sum_{j=1}^{n} q_j \theta_j \mathrm{e}^{\theta_j y} \mathrm{d}y + \int_{h-x}^{0} w_2(x+y) \sum_{j=1}^{n} q_i \theta_i \mathrm{e}^{\theta_i y} \mathrm{d}y \right. \\
+ \left. \int_{0}^{H-x} w_2(x+y) \sum_{j=1}^{m} p_j \eta_j \mathrm{e}^{-\eta_j y} \mathrm{d}y + \int_{H-x}^{+\infty} w_3(x+y) \sum_{j=1}^{m} p_j \eta_j \mathrm{e}^{-\eta_j y} \mathrm{d}y \right\rbrace .
\end{multline*}
Substituting expressions for $w_1(x)$, $w_2(x)$ and $w_3(x)$ yields, for $h<x<H$,
\begin{multline}\label{difx}
-\alpha \mathrm{e}^{\gamma x} = - \sum_{i=1}^{m+1} \mathrm{e}^{\beta_{i,\alpha+\rho}(x-H)} \omega_{i}^{0} \left( \frac{\sigma^{2}}{2} \beta_{i,\alpha+\rho}^{2} - \mu \beta_{i,\alpha+\rho}-(\lambda+\alpha+\rho) \right) \\
- \sum_{i=1}^{n+1} \mathrm{e}^{\gamma_{i,\alpha+\rho}(x-h)} \nu_{j}^{0} \left( \frac{\sigma^{2}}{2} \gamma_{i,\alpha+\rho}^{2} - \mu \gamma_{i,\alpha+\rho} - (\lambda+\alpha+\rho) \right) \\
+ \lambda \left\lbrace \int_{-\infty}^{h-x} \sum_{i=1}^{m+1} \mathrm{e}^{\beta_{i,\alpha}(x+y-h)} \omega_{i}^{L} \sum_{j=1}^{n} q_j \theta_j \mathrm{e}^{\theta_j y} \mathrm{d}y - \int_{h-x}^{0} \sum_{i=1}^{m+1} \mathrm{e}^{\beta_{i,\alpha+\rho}(x+y-H)} \omega_{i}^{0} \sum_{j=1}^{n} q_j \theta_j \mathrm{e}^{\theta_j y} \mathrm{d}y \right. \\
- \int_{h-x}^{0} \sum_{i=1}^{n+1} \mathrm{e}^{\gamma_{i,\alpha+\rho}(x+y-h)} \nu_{i}^{0} \sum_{j=1}^{n} q_j \theta_j \mathrm{e}^{\theta_j y} \mathrm{d}y - \int_{0}^{H-x} \sum_{i=1}^{m+1} \omega_{i}^{0} \mathrm{e}^{\beta_{i,\alpha+\rho}(x+y-H)} \sum_{j=1}^{m} p_j \eta_j \mathrm{e}^{-\eta_j y} \mathrm{d}y \\
- \int_{0}^{H-x} \sum_{j=1}^{n+1} \nu_{j}^{0} \mathrm{e}^{-\gamma_{j,\alpha+\rho}(x-h)} \sum_{j=1}^{m} p_j \eta_j \mathrm{e}^{-y(\eta_j+\gamma_{j,\alpha})} \mathrm{d}y \\
+ \left. \int_{H-x}^{+\infty} \sum_{j=1}^{n+1} \nu_j^{U} \mathrm{e}^{-\gamma_{j,\alpha}(x+y-H)} \sum_{j=1}^{m} p_j \eta_j \mathrm{e}^{-\eta_j y} \mathrm{d}y \right\rbrace \\
- c_{0} \left( \mathcal{L}-\alpha-\rho \right) \mathrm{e}^{\gamma x} \\
- \lambda \left\lbrace c_{L} \int_{-\infty}^{h-x} \mathrm{e}^{\gamma(x+y)} \sum_{j=1}^{n} q_j \theta_j \mathrm{e}^{\theta_j y} \mathrm{d}y - c_{0} \int_{-\infty}^{h-x} \mathrm{e}^{\gamma(x+y)} \sum_{j=1}^{n} q_j \theta_j \mathrm{e}^{\theta_j y} \mathrm{d}y \right. \\
- \left. c_{0} \int_{H-x}^{+\infty} \mathrm{e}^{\gamma(x+y)} \sum_{j=1}^{m} p_j \eta_j \mathrm{e}^{-\eta_j y} \mathrm{d}y + c_{U} \int_{H-x}^{+\infty} \mathrm{e}^{\gamma(x+y)} \sum_{j=1}^{m} p_j \eta_j \mathrm{e}^{-\eta_j y} \mathrm{d}y \right\rbrace .
\end{multline}
Since
$$
-c_{0} \left( \mathcal{L}-\alpha-\rho \right) \mathrm{e}^{\gamma x} = - \alpha \mathrm{e}^{\gamma x} ,
$$
and
$$
G \left( \beta_{i,\alpha+\rho} \right) - \alpha - \rho = G \left( \gamma_{j,\alpha+\rho} \right) - \alpha - \rho = 0 ,
$$
then computing the second derivative of Equation~\eqref{difx} with respect to $x$ yields~\eqref{dif2} and \eqref{dif2}.


The proof of Theorem~\ref{T:main} is complete.

%
%
\bibliographystyle{abbrv}
\bibliography{occ-time_options}

\end{document}